\documentclass{amsart}

\usepackage{amsfonts,amsmath,latexsym,amssymb}
\usepackage{amsthm}               
\usepackage{mathrsfs,upref} 
\usepackage{hyperref}
\usepackage{cleveref}
\usepackage[margin=3.5cm]{geometry}
\usepackage{color}
\usepackage{enumitem}

\newtheorem{theorem}{Theorem}[section]           
\newtheorem{lemma}[theorem]{Lemma}               
\newtheorem{corollary}[theorem]{Corollary}
\newtheorem{prop}[theorem]{Proposition}
\theoremstyle{definition}

\newtheorem{remark}[theorem]{Remark}

\newcommand{\obp}{\otimes^{\gamma}}
\newcommand{\N}{\mathbb N}
\def\la{\langle}
\def\ra{\rangle}

\begin{document}

\title[Arens regularity of projective tensor product of Schatten
  spaces]{On Arens regularity of projective tensor product of Schatten
  spaces}

\author{Lav Kumar Singh}

\address{Lav Kumar Singh\\ School of physical Sciences, Jawaharlal
  Nehru University, New Delhi }

\email{lav17\_sps@jnu.ac.in; lavksingh@hotmail.com}


\thanks{This research work was carried out with the financial support
  from the Counsil of Scientific and Industrial Research (Government
  of India) through a Junior Research Fellowship with
  No. \bf 09/263(1133)/2017-EMR-I}

\begin{abstract}
       We discuss Arens regularity of the projective tensor product of
       Banach algebras consisting of Schatten class operators. More
       precisely, exploiting \"Ulger's biregularity criterion, we
       prove that  for any Hilbert space $\mathcal H$, the Banach
       alegbras $S_p(\mathcal H)\otimes^\gamma S_q(\mathcal H)$ and
       $B(S_2(\mathcal H))\otimes^\gamma S_2(\mathcal H)$ are not
       Arens regular for every pair $1 \leq p, q \leq 2$; whereas,
       when $\mathcal H$ is separable and $S_2(\mathcal H)$ is
       equipped with the Schur product, then $S_2(\mathcal
       H)\otimes^\gamma S_2(\mathcal H)$ turns out to be Arens
       regular.
\end{abstract}
\maketitle
\section{Introduction}
       For any normed algebra $A$, Arens (in \cite{Arens}) defined two
       products $\Box$ and $\diamond$ on $A^{**}$ such that each
       product makes $A^{**}$ into a Banach algebra and the canonical
       isometric inclusion $J:A\to A^{**}$ becomes a homomorphism with
       respect to both the products. The normed algebra $A$ is said to
       be \textit{Arens regular} if the two products $\Box$ and
       $\diamond$ agree, i.e. $f\Box g=f\diamond g$ for all $f,g\in
       A^{**}.$ Over the years, people have witnessed some significant
       connection between Arens regularity and certain geometric
       properties of Banach spaces - see, for instance, \cite{LU,
         Neufang} and the references therein.

It is known that operator algebras and reflexive Banach algebras are
all Arens regular - see \cite{Palacios} and
\cite[1.4.2]{Palmer}. In particular,
$S_p(\mathcal H)$, the Banach algebra consisting of Schatten $p$-class
operators, being reflexive for $1< p< \infty$, is Arens regular. And,
even though $S_1(\mathcal H)$ is not reflexive, it is still Arens
regular - a proof can be found in \cite{Arikan}.

Given the importance of tensor products in the theory of Banach
algebras, among the various algebraic properties that people consider,
it is quite natural to analyze the Arens regularity of the projective
tensor product of Banach algebras. Some important results in this
direction appear, for instance, in \cite{Ulger, GI, Neufang, KR}
       
In fact, very recently, in \cite{Neufang}, Neufang settled a 40 years
old proplem regarding Arens regularity of the Varapolous algebra
$C(X)\obp C(Y)$. Actually, along with some other important results, he
proved, more generally, that for $C^*$-algebras $A$ and $B$, $ A \obp
B$ is Arens regular if and only if $A$ or $B$ has the Phillips
property. Prior to Neufang's result, \"Ulger, in his pioneering work
\cite{Ulger} in this direction, had shown that when $\ell^p$ is
equipped with pointwise multiplictaion, then $\ell^p\obp \ell^q$ is
Arens regular for $1 < p, q < \infty$. It was this setup that
motivated us to ask the same question for the non-commutative analogue
of $ \ell^p$ spaces, namely, the Banach algebra $S_p (\mathcal H)$
consisting of Schatten $p$-class operators on a Hilbert space
$\mathcal H$. Unlike the $\ell^p$ spaces, it turns out that the
projective tensor product of Schatten spaces is not Arens regular. We
achieve this by exploiting one of \"Ulger's technique from
\cite{Ulger} to show that $B(\mathcal K)\otimes^\gamma \mathcal K$ and
$S_p(\mathcal H)\otimes^\gamma S_q(\mathcal H)$ are not Arens regular
for $1\leq p, q \leq 2$, where $\mathcal K:= S_2(\mathcal H)$ for any
Hilbert space $\mathcal{H}$.\bigskip

Further, keeping one leg as a $C^*$-algebra, we show in
\Cref{BK-obp-K} that the Banach aglebra $B(S_2(\mathcal{H}))\obp
S_2(\mathcal{H})$ is not Arens regular. And then we deduce further
that the Banach algebras, $B_0(S_2(\mathcal{H}))\obp
S_2(\mathcal{H})$, $B(\mathcal K)\obp S_p(\mathcal H)$ and
$B_0(\mathcal K)\obp S_p(\mathcal H)$ and are all Arens irregular for
$1 \leq p \leq 2$.

On the other hand, when $\mathcal H$ is separable and $ S_2(\mathcal
H)$ is equipped with the Schur product, then, in the final section, we
show that the commutative Banach algebra $S_p(\mathcal
H)\otimes^\gamma S_q(\mathcal H)$ is Arens regular.

\section{Preliminaries}
\subsection{Schatten spaces} \( \)

       Let $\mathcal H$ be a Hilbert space and $T
   \in B(\mathcal H)$. For $1 \leq p < \infty$, the Schatten $p$-norm
   of $T$ is given by
   \[
   ||T||_p= \operatorname{Tr}(|T|^p),
   \]
where $\operatorname{Tr}$ denotes the canonical semi-finite positive trace on
$B(\mathcal H)_+$.  The {\em Schatten $p$-class operators} on
$\mathcal{H}$ are those $T\in B(\mathcal H)$ for which $||T||_p
<\infty$ and
   \[
   S_p(\mathcal H):= \{ T \in B(\mathcal{H}) : \|T\|_p  < \infty\}.
   \]
   $S_p(\mathcal H)$ is known to be an ideal in $B(\mathcal H)$ and
   $||\cdot||_p$ is a norm on $S_p(\mathcal H)$ which makes it a Banach
   $*$-algebra (with canonical adjoint involution). Whenever
   $1\leq p < q$, we have the inequality $||\cdot||_q\leq ||\cdot
   ||_p$, and hence the containment $S_p(\mathcal H)\subset
   S_q(\mathcal H)$.  Operators in $S_p(\mathcal H)$ are compact and
   $S_p(\mathcal H)$ contains all finite rank operators.  Detailed
   discussion about these facts can be found in
   \cite{Palmer}.

\subsection{Arens regularity}
       Let $A$ be a normed algebra. For the sake of convenience, we
       quickly recall the definitions of the two products $\Box$ and
       $\diamond$ mentioned in the Introduction.  For $a \in A$,
       $\omega\in A^*$, $f\in A^{**}$, consider the functionals
       $\omega_a, {}_{a}\omega\in A^*$, $\omega_f,{}_f\omega\in
       A^{**}$ given by $w_a=(L_a)^*\omega$,
       ${}_a\omega=(R_a)^*\omega$; $\omega_f(a)=f({}_a\omega)$ and
       ${}_f\omega(a)=f(\omega_a)$. Then, for $f,g\in A^{**}$ the
       operations $\Box$ and $\diamond$ are given by $(f\Box g)
       (\omega)= f({}_g\omega)$ and $(f\diamond
       g)(\omega)=g(\omega_f)$ for all $\omega \in A^*$. As recalled
       in the Introduction, $A$ is said to be Arens regular if the
       products $\Box$ and $ \diamond$ are same on $A^{**}$.
       
       Thanks to \"Ulger, there is a very useful equivalent
       characterization of Arens regularity in terms of some
       properties of bilinear maps.
       
       A bounded bilinear form $m:X\times Y\to Z$, where $X,Y$ and $Z$ are
       normed spaces, is called {\em Arens regular} if the induced bilinear forms
       $m^{***}:X^{**}\times Y^{**}\to Z^{**}$ and $m^{r***r}:X^{**}\times
       Y^{**}\to Z^{**}$ are same, where $m^*:Z^*\times X\to Y^*$ is given by
       $m^*(f,x)(y)=\langle f,m(x,y)\rangle$, $m^r:Y\times X\to Z$ is
       given by $m^r(y,x)=m(x,y)$, $m^{**}:=(m^*)^*$ and so on.
       
       A normed algebra $A$ is known to be Arens regular if and only
       if the multiplication $A \times A \to A$ is Arens regular in
       the above sense.

       \subsection{Projective tensor product and its Arens reguality}$\ $
       Let $A$ and $B$ be Banach algebras. There is a natural
       multiplication on their algebraic tensor product $A \otimes B$
       given by $(x_1\otimes y_1)(x_2\otimes y_2)=x_1x_2\otimes
       y_1y_2$ and extended appropriately to all tensors. And there
       are various ways to impose a normed algebra structure on $A
       \otimes B$.
       
       We will be primarily interested in the Banach space projective tensor
       product $\obp$.
       
       For $u \in A\otimes Y$, its projective norm is given by
       \[
       \Vert u \Vert_\gamma =\inf\left\{\sum_{i=1}^{n}\Vert x\Vert
       \Vert y\Vert~:~u=\sum_{i=1}^{n}x_i\otimes y_i\right\}.
       \]
       This norm turns out to be a \textit{cross norm}, i.e., $\|x
       \otimes y\|_{\gamma} = \|x\| \|y\|$ and the completion of the
       normed algebra $A\otimes B$ with respect to this norm is a
       Banach algebra and is denoted by $A\otimes^\gamma B$.  \"Ulger
       gave a very useful characterization of Arens regularity of the
       projective tensor product of Banach algebras in terms of
       biregularity of bilinear forms.
       
       Recall that, for Banach algebras $A$ and $B$, a bilinear form
       $m: A \times B \rightarrow \mathbb C$ is said to be {\em
         biregular} if for any two pairs of sequences $(a_i),
       (\tilde{a}_j)$ and $(b_i), (\tilde{b}_j)$ in the closed unit
       balls of $A$ and $B$, respectively, one has
       \begin{equation}\label{double-limits}
       \lim_i \lim_j m(a_i \tilde{a}_j, b_i \tilde{b}_j) = \lim_j
       \lim_i m(a_i \tilde{a}_j, b_i \tilde{b}_j)
       \end{equation}
       provided that these limits exist - see \cite[Definition
         3.1]{Ulger}.  Our results will depend heavily on the
       following characterization provided by \"Ulger.
       \begin{theorem}\cite{Ulger}\label{ulger}
       	Let $A$ and $B$ be  Banach algebras. Then, their
       	projective tensor product $A \obp B$ is Arens regular if and only if
       	every bilinear form $m: A \times B \rightarrow \mathbb C$ is biregular.
       \end{theorem}

\section{Arens regularity of some bilinear maps}

We will see ahead that $A \obp
  B$ is Arens regular if $A$ is so and $B$ is finite
  dimensional, but, prior to that, we observe that, for the space of
  Hilbert-Schmidt operators, some specific bilinear maps behave well
  when one leg is finite dimensional, as in:

\begin{prop}\label{finite-biregular}
  Let $\mathcal K_1=S_2(\mathcal H_1)$ and $\mathcal K_2=S_2(\mathcal
  H_2)$, where $\mathcal H_1$ and $\mathcal H_2$ are two Hilbert
  spaces. Suppose $\mathcal{H}_1$ or $\mathcal{H}_2$ is finite
  dimensional and $m:\mathcal K_1\times \mathcal K_2\to \mathbb C$ is
  a bilinear form. Then, for every pair of sequences $\{S_i\},\{\tilde
  S_j\}$ and $\{T_i\},\{\tilde T_j\}$ in the unit balls of
  $\mathcal K_1$ and $\mathcal K_2$, respectively, if the iterated
  double limits (as on both sides of \Cref{double-limits}) exist then
  they are equal to $m(S\tilde S, T\tilde T)$, where $S,\tilde S$ and
  $T,\tilde T$ are some weak limits of the correspondig sequences.

  In particular, every bilinear form $m:\mathcal K_1\times
  \mathcal K_2\to \mathbb C$ is biregular.

\end{prop}

  \begin{proof}
First, note that, due to Reisz Representation Theorem, there exists a
conjugate linear continuous map $\varphi:\mathcal K_1\to \mathcal K_2$
such that $m(S,T)=\left<T,\varphi(S)\right>$ for all $S$, $T$.

    Suppose that $\mathcal H_2$ is a finite dimentional Hilbert space
    with an orthonormal basis $\{e_r\}_{r=1}^n$. Let $\{S_i\},\{\tilde
    S_j\}$ and $\{T_i\},\{\tilde T_j\}$ be a two pairs of sequences in
the unit balls of $\mathcal K_1$ and $\mathcal K_2$, respectively, such
    that the iterated limits $\lim_i\lim_j m(S_i\tilde S_j,T_i\tilde
    T_j)$ and $\lim_j\lim_i m(S_i\tilde S_j,T_i\tilde T_j)$ exist. Then, we have
	\begin{equation}\label{two-limits}
	\begin{split}
	\lim_i\lim_j m(S_i\tilde S_j,T_i\tilde T_j)&=\lim_i\lim_j
        \left<T_i\tilde T_j,\varphi(S_i\tilde
        S_j)\right>\\ &=\lim_i\lim_j ~\text{Tr}\left(\varphi(S_i\tilde
        S_j)^* T_i\tilde T_j\right)\\ &=\lim_i\lim_j
        \sum_{r=1}^n\left<T_i\tilde T_j e_r,\varphi(S_i\tilde
        S_j)e_r\right>.
	\end{split}
	\end{equation}
        
	Since the unit ball of a Hilbert space is weakly compact and
        the unit ball of the finite dimentional space $\mathcal K_2$
        is compact, the sequences $\{S_i\},\{\tilde S_j\}$ and
        $\{T_i\},\{\tilde T_j\}$ have weakly and norm convergent
        convergent subsequences, respectively.  Let
        $\{S_{i_k}\},\{\tilde S_{j_l}\}$ and $\{T_{i_k}\},\{\tilde
        T_{j_l}\}$ be the corresponding pair of subsequences
        converging to $S,\tilde S$ and $T,\tilde T$ in weak*
        (equivalently, weak) topology and norm topology on $\mathcal
        K_1$ and $\mathcal K_2$, respectively. Then, we obtain
	\begin{equation}\label{E6}
	\lim_i\lim_j m(S_i\tilde S_j,T_i\tilde
	T_j)=\lim_{i_k}\lim_{j_l} \sum_{r=1}^n\left<T_{i_k}\tilde
	T_{j_l} e_r,\varphi(S_{i_k}\tilde
	S_{j_l})e_r\right>. 
	\end{equation}
        
	Hence, without loss of generality, we may assume that
        $\{S_i\}\overset{w}{\to} S$, $\{\tilde
        S_j\}\overset{w}{\to}\tilde S$ in the weak topology of
        $\mathcal K_1$ and $\{T_i\}\to T$, $\{\tilde T_j\}\to \tilde
        T$ in the norm (HS) topology of $\mathcal K_2$.  The limits and
        summation are interchangeable because summation is over finite
        index set. Hence, we have
	\begin{equation} \label{E7}
	\lim_i\lim_j m(S_i\tilde S_j,T_i\tilde
	T_j)=\sum_{r=1}^n\lim_{i}\lim_{j}\left<T_{i}\tilde T_{j}
	e_r,\varphi(S_{i}\tilde S_{j})e_r\right>.
	\end{equation}

	Next, we make few claims, whose proofs will be provided later:\\
	{\bf \underline{Claims:}}

	\begin{enumerate}[label=(C\arabic*)]
		\item $\tilde T_{j}e_r \to \tilde Te_r $ in norm for
                  every $r$.
		\item $T_{i}\tilde T_{j}e_r \to T_{i}\tilde Te_r$ in
                  norm for every $i$ and $r$.
		\item $S_{i}\tilde S_{j} \overset{w}{\to} S_{i}\tilde
                  S$ for every $i$.
		\item $\varphi(S_{i}\tilde S_{j})e_r \to
                  \varphi(S_{i}\tilde S)e_r$ in norm for every $r$ and
                  $i$.
		\item $S_i\tilde S \to S \tilde S$ weakly.
		\item $\varphi(S_i\tilde S)e_r\to \varphi(S\tilde
                  S)e_r$ in norm for every $r$.
	\end{enumerate}

	Using (C2) and (C4) and the fact that every inner product is
        continuous in both cordinates (with respect to the norm
        topology), \Cref{E7} gives us
	\begin{equation} \label{E8}
	\begin{split}
	\lim_i\lim_j m(S_i\tilde S_j,T_i\tilde
	T_j)&=\sum_{r=1}^n\lim_i\left<\lim_{j}T_{i}\tilde
	T_{j}e_r,\lim_{j}\varphi(S_{i}\tilde S_{j})e_r
	\right>\\
        &=\sum_{r=1}^n \lim_i\left<T_i\tilde T
	e_r,\varphi(S_i\tilde S)e_r\right>.
	\end{split}
	\end{equation}
	Using the continuity of operators $T_i$ and (C6) combined with
        the norm continuity of inner product in \Cref{two-limits}, we
        obtain
	\begin{equation*}
	\begin{split}
	\lim_i\lim_j m(S_i\tilde S_j,T_i\tilde
	T_j)&=\sum_{r=1}^n\left<\lim_i T_i\tilde Te_r,\lim_i
	\varphi(S_i\tilde S)e_r\right>\\&=\sum_{r=1}^n\left<T\tilde
	Te_r,\varphi(S\tilde S)e_r\right>\\&=m(S\tilde S,T\tilde T).
	\end{split}
	\end{equation*}
	Repeating the same process with limits interchanged, we obtain
	\begin{equation*}
	\begin{split}
	\lim_j\lim_i m(S_i\tilde S_j,T_i\tilde
        T_j)&=\sum_{r=1}^n\left<T\tilde Te_r,\varphi(S\tilde
        S)e_r\right>\\&=m(S\tilde S,T\tilde T).
	\end{split}
	\end{equation*} 
This also proves that the bilinear form $m$ is biregular.\smallskip

We now prove the claims made above:\smallskip

(C1) Since $(\tilde T_{j_k}-\tilde T)\to 0$ in the norm topology of
$\mathcal K_2$, it must follow that $(\tilde T_{j_k}-\tilde T)\to 0$
in the operatorn norm of $B(\mathcal H_2)$(because Hilber-Schmidt norm
dominates the operator norm). Hence $\tilde T_{j_k}e_\alpha\to \tilde
T e_r$ for each $r$.

(C2) Follows from the sequential
continuity of $T_{i_k}$ as a bounded operator and C1.

(C3) Since $\tilde S_{j_k}\overset{w}{\to}\tilde S$, we have
$\left<\tilde S_{j_k}, P\right>\to \left<\tilde S,P\right>$ for each
$P\in \mathcal K_1$.  Hence $\left<\tilde S_{j_k}, S_{i_k}^*P \right>
\to \left<\tilde S,S_{i_k}^*P\right>$; so that $\left< S_{i_k}\tilde
S_{j_k}, P \right> \to \left<S_{i_k}\tilde S,P\right>$ for each $P\in
\mathcal K_1$.  Hence $S_{i_k}\tilde S_{j_k}\overset{w}{\to}
S_{i_k}\tilde S$.

(C4) Any bounded linear operator between Hilbert spaces is weak-weak
continuous and weak topology on $\mathcal K_2$ is nothing but norm
topology (due to $\mathcal K_2$ being finite dimensional), hence due
to C3 we obtain $\varphi(S_{i_k}\tilde S_{j_k})\to
\varphi(S_{i_k}\tilde S)$ in norm of $\mathcal K_2$. And hence we have
the following norm convergence in $\mathcal
H_2$. $$\varphi(S_{i_k}\tilde S_{j_k})e_\alpha\to
\varphi(S_{i_k}\tilde S)e_\alpha $$

(C5) Consider $\left<S_i\tilde S,
P\right>=\operatorname{Tr}(P^*S_i\tilde S)$ for some $P\in\mathcal
K_1$. We know that product of two Hilbert Schmidt operators is a trace
class operator. Hence $P^*S_i\in S_1(\mathcal H_1)$. Also for any
$A\in \mathcal S_1(H)$ and $B\in\mathcal B(H)$, we have
$\operatorname{Tr}(AB)=\operatorname{Tr}(BA).$ So,
\[
	\left< S_i\tilde S,P\right> = \operatorname{Tr}(P^*S_i\tilde S)=\operatorname{Tr}(\tilde SP^*S_i) \left<S_i,P\tilde S^*\right>.
\]
And, since $S_i\overset{w}{\to}S$, we see that
\[
\left<S_i,P\tilde
	S^*\right>\to \left<S,P\tilde
	S^*\right>=\operatorname{Tr}(\tilde
	SP^*S)=\operatorname{Tr}({P^*S\tilde S})=\left<S\tilde
	S,P\right>.
        \]
        Hence, $\left<S_i\tilde S,P\right>\to
	\left<S\tilde S,P\right>$ for each $P\in \mathcal
	K_1$, i.e., $S_i\tilde S\overset{w}{\to}S\tilde S .$

        (C6) This
	again follows from (C5) and the weak-norm continuity of
	$\varphi$.
\end{proof}

The following type of Bilinear forms play an important role in upcoming results in the next section. We first note in the next theorem that these maps are Arens regular. Later on we will show that when $\mathcal H$ is replaced by Hilbert space of Hilbert-Schmidt operators, then such forms, though Arens regular, are not "Biregular" in general.  
\begin{prop}\label{T1}
	Let $\mathcal H$ be a Hilbert space and $l :B(\mathcal H)\times
	\mathcal H\to \mathcal H$ be the bounded bilinear map given by
	$l(T,\zeta)=T(\zeta) $. Then, $l$ is Arens regular. Also, the bilinear form $m:B(\mathcal H)\times \mathcal H\to\mathbb C$, defined as $m(T,\zeta)=\left<T(\zeta),\beta\right>$ for some $\beta\in \mathcal H$, is Arens regular.
\end{prop}
\begin{proof}
	Let $V\in B(\mathcal H)^{**}$ and $F\in \mathcal
	H^{**}$. Since Hilbert spaces are reflexive, $F=J_{\xi}$ for
	some $\xi \in \mathcal H$. We need to show that
	$l^{***}(V,F)=l^{r***r}(V,F)$. Let $f\in \mathcal
	H^*$. Then, by Reisz representation Theorem, $f=~\langle
	~\cdot~, \eta\rangle $ for a unique $\eta \in \mathcal
	H$. Further, we have
	\begin{equation}\label{E1}
	l^{***}(V,F)(f)=\langle V,l^{**}(F,f)\rangle.
	\end{equation} 
	Note that $l^{**}(F,f)(T)=~\la F,l^*(f,T) \ra $ and
	$l^*(f,T)(\zeta)=f(T\zeta) = \la l(T,\zeta), \eta \ra $ for
	all $T \in \mathcal{H}^{**}$ and $\zeta \in \mathcal{H}$; so,
	\begin{eqnarray}\label{E2}
	l^{**}(F,f)(T) &=& \langle F,\la l(T,~\cdot~),\eta \ra \rangle
        \nonumber\\ &=& J_\xi(\la l(T,~\cdot~),\eta\ra )
        \nonumber\\ &= & \la l(T,\xi),\eta\ra \nonumber .
	\end{eqnarray}
	for all $T \in \mathcal{H}^{**}$; so, $  l^{**}(F,f) =  \la
	l(~\cdot~,\xi),\eta\ra$. And,  similarly, we obtain
	\begin{equation}\label{E3}
	l^{r***r}(V,F)(f)=l^{r***}(F,V)(f)=\langle
        F,l^{r**}(V,f)\rangle=l^{r**}(V,f)(\xi)
	\end{equation}
	and $ l^{r**}(V,f)(\xi)= \la V,l^{r*}(f,\xi)\ra $.
	Also, note that 
	\[
	l^{r*}(f,g)(S)  = f(l(S,\xi))\\  = \langle l(S,\xi),\eta \rangle
	\]
	for all $S \in B(\mathcal H)$; so,
        \begin{equation}\label{E4} 
	l^{r**}(V,f)(\xi) =\langle V,\langle l(~\cdot~,\xi ),\eta
        \rangle \rangle.
	\end{equation}
	Thus, from Equations (\ref{E1}),(\ref{E2}),(\ref{E3}) and
        (\ref{E4}), we conclude
        that $$l^{***}(V,F)(f)=l^{r***r}(V,F)(f)=\langle V,\la
        l(~\cdot~,\xi),\eta \rangle \rangle$$ for all $f\in \mathcal
        H^*$. Thus, $l^{***}(V,F)=l^{r***r}(V,F)$. We conclude that
        $l^{***}=l^{r***r}$ and hence $l$ is Arens regular. Arens regularity of $m$ can be proved in similar fashion.
\end{proof}

       \section{Arens regularity of projective tensor product of Schatten spaces}
   
   We will be particularly interested in $S_2(\mathcal H)$, the space of
   Hilbert-Schmidt operators on $\mathcal{H}$. It is known to be a
   Hilbert space with respect to the inner product $\langle A,B\rangle :=
   \operatorname{Tr}(B^*A)$, and is also denoted by $\mathcal K$.  Being
   reflexive, $\mathcal{K}$ is an Arens regular Banach algebra.

   \"Ulger had proved in \cite[Corr. 4.7]{Ulger} that
   $\ell^2\otimes^\gamma\ell^2$ is Arens regular (where $\ell^2$ is
   equipped with the pointwise product).  One would guess that the
   same should hold for $\mathcal K \otimes^\gamma \mathcal K$ as
   well. However, we have the following: 
   
   \begin{theorem}\label{ktimes k}
   	Let $\mathcal H_1$ and $\mathcal H_2$ be two infinite
        dimensional Hilbert spaces. Let $\mathcal K_1:=S_2(\mathcal
        H_1)$ and $\mathcal K_2:=S_2(\mathcal H_2)$. Then, the Banach
        algebra $\mathcal{K}_1\otimes^\gamma \mathcal{K}_2$ is not
        Arens regular.
   \end{theorem}
   \begin{proof}
   	Without the loss of generality, we can assume that $\mathcal
        H_1$ and $\mathcal H_2$ are separable infinite dimensional
        Hilbert spaces.  Fix any two orthonormal bases $\{e_i\}_{i\in
          \mathbb N}$ and $\{f_j\}_{j\in \mathbb N}$ for $\mathcal
        H_1$ and $\mathcal H_2$, respectively. Consider the pair of
        sequences $\{S_i\}$ and $\{\tilde S_j\}$ in $\mathcal K_1$
        given by $S_i=e_i\otimes e_1$ and $\tilde S_j=e_1\otimes e_j$.
        Note that $S_i\tilde S_j=e_i\otimes e_j$ for all $i, j \in
        \mathbb N$.  Recall that the association
   	\[
   	\mathcal H {\otimes} \mathcal H \ni \xi \otimes \eta \mapsto
   	\theta_{\xi, \eta} \in S_2(\mathcal{H})
   	\]
   	extends to a unitary from $\mathcal H \bar{\otimes} \mathcal
        H$ onto $S_2(\mathcal H)$ for any Hilbert space
        $\mathcal{H}$. We will take the liberty to use this
        identification without any further mention. In particular, the
        set $\{S_i\tilde S_j\}_{i,j\in \mathbb N}$ thus forms an
        orthonormal basis for $\mathcal K_1$.
   	
   	   	Now, let $F_{i,j}:=f_i\otimes f_j$ whenever $j\leq i$ and $0$
   	otherwise.  Define $\varphi:\mathcal K_1\to\mathcal K_2$ by
   	\[
   	\varphi\big(\sum c_{i,j}S_i\tilde S_j\big)=\sum \bar{c}_{i,j}F_{i,j},
   	\ \sum_{i,j}|c_{i,j}|^2 < \infty.
   	\]
   	Clearly, $\varphi$ is a conjugate linear (contractive) continuous
   	map. Thus, the map $m: \mathcal{K}_1\times \mathcal{K}_2 \rightarrow \mathbb C$
   	given by $m(S, T) = \langle T, \varphi(S)\rangle $ is a bounded
   	bilinear form. We assert that it is not biregular.
   	
   	Towards this end, let $T_i:=f_i\otimes f_1$ and $\tilde T_j:=f_1\otimes f_j$ for all $i, j
   	\in \mathbb N$. Note that both the pairs of sequences $\{T_i\},\{\tilde
   	T_j\}$ and $\{S_i\},\{\tilde S_j\}$ are bounded sequences in the unit
   	balls of $\mathcal K_2$ and $\mathcal K_1$, respectively. Observe
   	that
   	\begin{equation*}
   	\begin{split}
   	m(S_i\tilde S_j,T_i\tilde T_j)&= \left<T_i\tilde
   	T_j,\varphi(S_i\tilde
   	S_j)\right>\\&= \operatorname{Tr}\left(\varphi(S_i\tilde
   	S_j)^*T_i\tilde T_j\right)\\&=\sum_{r\in \mathbb N}\left<T_i\tilde
   	T_j (f_r),\varphi(S_i\tilde S_j)(f_r)\right>
   	\end{split}
   	\end{equation*}
   	for all $i, j \in \mathbb N$. Thus,
   	\begin{equation*}
   	\begin{split}
   	\lim_i\lim_j m(S_i\tilde S_j,T_i\tilde T_j)&=\lim_i\lim_j
   	\sum_{r\in \mathbb N}\left<T_i\tilde T_j (f_r),\varphi(S_i\tilde
   	S_j)(f_r)\right>\\&=\lim_i\lim_j
   	\left<f_i,F_{i,j}(f_j)\right>\\&= 0; \text{ and }
   	\end{split}
   	\end{equation*}
   	\begin{equation*}
   	\begin{split}
   	\lim_j\lim_i m(S_i\tilde S_j,T_i\tilde
   	T_j)&=\lim_j\lim_i \sum_{r\in I}\left<T_i\tilde T_j
   	(f_r),\varphi(S_i\tilde
   	S_j)(f_r)\right>\\&=\lim_j\lim_i
   	\left<f_i,F_{i,j}(f_j)\right>\\&= 1.
   	\end{split}
   	\end{equation*}
   	Thus, the bilinear form $m$ is not biregular and hence, by
        \Cref{ulger}, $\mathcal K_1\otimes^\gamma \mathcal K_2$ is not
        Arens regular.\smallskip
   \end{proof}
   
   \begin{corollary}
   	For $1\leq p,q\leq 2$, the Banach algebra $S_p(\mathcal
   	H_1)\otimes^\gamma S_q(\mathcal H_2)$ is not Arens regular.
   \end{corollary}
   \begin{proof}
   	Since $S_p(\mathcal H)\subset S_2(\mathcal H)$ for every
   	$1\leq p\leq 2$, the bilinear form $m: S_p(\mathcal H_1)\times
   	S_q(\mathcal H_2)\to \mathbb C$ defined as
   	$m(S,T)=\left<T,\varphi(S)\right>$ for some conjugate linear bounded
   	map $\varphi:\mathcal H_1\to \mathcal H_2$ is still  bounded  because
   	\[
   	|m(S,T)|\leq |\phi||\, ||T||_2||S||_2\leq ||\phi||\, ||S||_p
   	|||T||_q\ 
   	\]
   	$\text{for all}\ (S, T) \in S_p(\mathcal H_1)\times S_q(\mathcal
   	H_2).$ For the same choices of pairs of sequences and the
   	conjugate linear map $\varphi$ as in previous theorem, $m$ is
   	not biregular.
   \end{proof}

We can thus deduce  the following well known result:
\begin{corollary}   
Let $\mathcal H_1$ and $\mathcal H_2$ be Hilbert spaces. Then, the
Banach algebra $S_p(\mathcal H_1)\otimes^\gamma S_q(\mathcal H_2)$ is
not reflexive for every pair $ 1\leq p, q \leq 2$.
\end{corollary}

   The following must be a folklore. We include a proof because of
   lack of knowledge of a proper reference.
   \begin{prop}\label{finite}
Let $A$ be an Arens regular Banach algebra and $B$ be a finite
dimensional Banach algebra. Then, $A \obp B$ is Arens regular.
     \end{prop}
   \begin{proof}
Since $B$ is finite dimensional, for each $b''\in B^{**}$ the maps
$\tau_{b''}:B^{***}\to B^{***}$ and $_{b''}\tau$ defined as $x''\mapsto
x''b''$ and $x''\mapsto b''x''$ are compact. Hence, by
\cite[Th. 4.5]{Ulger}, each bilinear form $m:A\times B\to \mathbb C$
is biregular; so, by \Cref{ulger}, $A\obp B$ is Arens regular.
   \end{proof}

   \begin{corollary}
Let $\mathcal K_1=S_2(\mathcal H_1)$ and $\mathcal
        K_2=S_2(\mathcal H_2)$, where $\mathcal H_1$ and $\mathcal
        H_2$ are two Hilbert spaces. Then, $\mathcal K_1\otimes^\gamma
        \mathcal K_2$ is Arens regular if and only if either $\mathcal
        H_1$ or $\mathcal H_2$ is finite dimensional.
\end{corollary}
   \begin{proof}
Necessity follows from \Cref{ktimes k}. We saw a proof of sufficiency
in \Cref{finite-biregular} and it also follows from
\Cref{finite}. Alternately, if $\mathcal H_2$ is finite dimensional
then every map $\mathcal K_1\to \mathcal K_2$ is compact. This implies
that $\mathcal K_1\otimes^\gamma \mathcal K_2$ is reflexive (by
\cite[Th. 4.21]{Ryan}) and hence Arens regular.
     \end{proof}

Being a $C^*$-algebra, $B(\mathcal{H})$ is Arens regular for any
Hilbert space $H$. One would guess that $B(\mathcal K)\obp
\mathcal K$ should also be Arens regular, but it turns out to be the
opposite.

\begin{theorem}\label{BK-obp-K}
Let $\mathcal{H}$ be an infinite dimensional Hilbert space and
$\mathcal K$ denote the Banach algebra $S_2(\mathcal H)$. Then, the
Banach algebra $B(\mathcal K)\obp \mathcal K$ is not Arens regular.
\end{theorem}
\begin{proof}	
	For $\xi,\eta\in \mathcal{H}$, let $\theta_{\xi,\eta}$ be the rank-one
	operator on $\mathcal{H}$ given by $\theta_{\xi,\eta}(\gamma) =
	\langle \gamma,\eta\rangle \xi$ for $\gamma\in \mathcal{H}$. Clearly,
	\[
	\theta_{\xi, \eta}\circ \theta_{\zeta, \delta} = \langle \zeta, \eta\rangle
	\theta_{\xi, \delta}.
	\]

	Fix a unit vector $\xi_0\in \mathcal{H}$ and define a bilinear form
	$m:B(\mathcal K)\times \mathcal K\to \mathbb C$ as
	$$
	m(T,A)=\langle
	T(A),\theta_{\xi_0, \xi_0}\rangle,\  T \in B(\mathcal K), A \in \mathcal{K}.
	$$
	Clearly, $m$ is bounded. We show
	that $m$ is not biregularr, which will then, by \Cref{ulger}, imply that
	$B(\mathcal K)\obp \mathcal K$ is not Arens regular.

	Let $(e_i)$ be an orthonormal sequence in $\mathcal H$ with $e_1 = \xi_0$. Set
	$S_i = \theta_{e_1, e_i} \in B(\mathcal H)$ and let $R_j$ denote 
	the orthogonal projection onto the span of
	$\{e_1,e_2,\ldots,e_j\}$. Then,
	\begin{equation}\label{A}
	\lim_i \lim_j \langle S_iR_j(e_i),e_1\rangle = \lim_i
	\langle S_i(e_i),e_1\rangle = \lim_i \langle e_1,e_1\rangle = 1.
	\end{equation}
	And, on the other hand, we have
	\begin{equation}\label{B}
	\lim_j \lim_i \langle S_iR_j(e_i),e_1\rangle = \lim_j \lim_i
	\langle (R_j(e_i),e_i\rangle \langle e_1,e_1\rangle = 0.
	\end{equation}
	
	Now, let $\tilde A_j := \theta_{e_1,e_1}$ and $A_i :=
	\theta_{e_i,e_1}$; so that $A_i \tilde A_j = A_i$ for all $i,
	j\in \mathbb N$. Clearly, $\|A_i\|, \|\tilde{A}_j\| \leq 1 $ for all
	$i, j \in \mathbb N$.
	
	Further, for each $i, j \in \mathbb N$, define $\tilde{T_j}, T_i:
	S_2(\mathcal H) \rightarrow S_2(\mathcal H) $ by
	\[
	T_i(A) = \theta_{S_i(A(e_1)),e_1},\ \tilde T_j(A) =
	\theta_{R_j(A(e_1)), e_1} \ \text{for}\ A \in S_2(\mathcal H).
	\] One can easily check that  $\|T_i\|, \|\tilde{T}_j\| \leq 1$ for all $i, j \in \mathbb N$.          Then, we have
	$$ \tilde T_j (A) (e_1) = \theta_{R_j(A(e_1)), e_1} (e_1) =
	R_j(A(e_1))\ \text{for all}\ A \in S_2(\mathcal H)).
	$$ In particular, $ T_i \tilde T_j(A) = \theta_{S_i
		R_j(A(e_1)), e_1}$ for all $A \in S_2(\mathcal H)$, which
	then yields $ T_i \tilde T_j (A_i \tilde A_j) = T_i \tilde
	T_j(\theta_{e_i,e_1}) = \theta_{S_iR_j(e_i),e_1} $ for all $i,
	j \in \mathbb N$. In particular, we have 
	\begin{equation}\label{C}
	m(T_i \tilde T_j,  A_i \tilde A_j) =          \Big\langle  T_i \tilde T_j (A_i \tilde A_j), D \Big\rangle
	= \operatorname{Tr}\Big( \theta_{S_i R_j(e_i),e_1} \Big)
	= \langle S_i R_j(e_i) , e_1\rangle\ 
	\end{equation}
	for all $i,j\in \mathbb N$.
	Thus, Equations (\ref{A}), (\ref{B})
	and (\ref{C}) tell us that $m$ is not biregular.
\end{proof}
For any Hilberts space $\mathcal{H}$, $B_0(\mathcal H)$, the space of
compact operators on $\mathcal{H}$, being a $C^*$-algebra, is Arens
regular. The preceding technique also shows that $B_0(\mathcal K)\obp \mathcal
K$ is not Arens regular.

\begin{corollary}\label{B0K-obp-K}
Let $\mathcal H$ and $\mathcal K$ be as in \Cref{BK-obp-K}. Then, the
Banach algebra $B_0(\mathcal K)\obp \mathcal K$ is not Arens regular.
\end{corollary}
\begin{proof}
	Notice that the operators $T_i$ and $\tilde T_j$ constructed
	in \Cref{BK-obp-K} are finite rank operators on $\mathcal K$
	and hence compact. Thus, the same pairs of sequences
	$\{T_i\},\{\tilde T_j\}$ and $\{A_i\},\{\tilde A_j\}$ tell us
	that the bounded bilinear form $m:B_0(\mathcal K)\times
	\mathcal K\to \mathbb C$, defined by $m(T,A)=\langle T(A),
	D\rangle$ for some fixed $D \in \mathcal K$, is not
	biregular. Rest is again taken care of by \Cref{ulger}.
\end{proof}

\begin{corollary}\label{BK-obp-SpH}
Let $\mathcal H$ and $\mathcal K$ be as in \Cref{BK-obp-K}. Then, the Banach
algebras $B(\mathcal K)\obp S_p(\mathcal H)$ and $B_0(\mathcal K)\obp
S_p(\mathcal H)$, for $1\leq p \leq 2$, are not Arens regular.
\end{corollary}
\begin{proof}
	Notice that the bilinear form $m:B(\mathcal K)\times S_p(\mathcal
	H)\to \mathbb C$ defined as $m(T,A)=\left<T(A),D\right>$ is still a
	well defined bounded bilinear form because $S_p(\mathcal H)\subset
	S_2(\mathcal H)$ (as  $||.||_2\leq ||.||_p$) for $p\leq
	2$.  Hence, for the same choice of the pairs of bounded sequences as
	in the preceding theorem, $m$ is not biregular.
\end{proof}

\begin{corollary}\label{BH-obp-SpH}
Let $\mathcal H$ and $\mathcal K$ be as in \Cref{BK-obp-K}. Then, the
Banach algebra $B(\mathcal H)\obp S_p(\mathcal H)$ is not Arens
regular for every $1\leq p\leq 2$.
\end{corollary}

\begin{proof}
	Note that $\mathcal H$ and $S_2(\mathcal H)$ (=:$\mathcal K$) are
	isomophic Hilbert spaces because dimension of $\mathcal H$ and
	$S_2(\mathcal H)$ are same. (If $\{e_\alpha\}_{\alpha \in I}$ is an
	orthonormal basis for $\mathcal H$ then $\{e_i\otimes e_j\}_{i,j\in
		I}$ is an orthonormal basis for $ S_2(\mathcal H)$ ). Rest follows
	from \Cref{BK-obp-SpH}.
\end{proof}

\begin{remark}
		 We have thus observed that if $\mathcal A$ is any of the
		Banach alegbras $B(\mathcal K)$ or $S_p(\mathcal H)$ for
		$p\in[1,\infty)$, and $\mathcal B$ is any of the Banach algebras
		$S_1(\mathcal H)$ or $S_2(\mathcal H)$, then $\mathcal
		A\otimes^\gamma \mathcal B$ is not Arens regular.
	\end{remark}

\begin{remark}
  Taking $\ell^p$ with pointwise multiplication, it was shown in
		\cite[Corrollary 4.7]{Ulger2} that $\ell^p\otimes^\gamma \mathcal A$
		is Arens regular if and only if $\mathcal A$ is Arens regular. Even
		though $S_p(\mathcal{H})$ is Arens regular, we have thus observed
		that a similar characterization does not hold for $S_p(\mathcal
		H)\otimes^\gamma \mathcal A$.
                \end{remark}

We concluded this section with some immediate consequences.
\begin{remark}
  \begin{enumerate}
\item Let $\mathcal{H}$ be an infinite dimensional Hilbert
  space. Although $\mathcal K$ is an operator algebra (see
  \cite{Blecher}), the Banach algebra $\mathcal K\otimes^\gamma
  \mathcal K$ is not an operator algebra (by \Cref{ktimes k}), since operator algebras being a subalgebra of $B(H)$ must be Arens regular (subalgebra of a Arens regular Banach alegbra is Arens regular).
\item	 The bilinear form $m$ defined in Theorem \ref{BK-obp-K}
         serves as an example of a form which is Arens regular (see
         Theorem \ref{T1}) but not biregular (as proved in
         \Cref{BK-obp-K}), although one of the algebra is a unital
         $C^*$-algebra.
\end{enumerate}
  \end{remark}

\section{Arens regularity of tensor product of Hilbert-Schmit spaces with Schur product}
   When $\mathcal H$ is infinite dimensional and separable, then there
   is another important multiplication on $S_p(\mathcal H)$ which is
   given by pointwise multiplication of matrices of operators with
   respect to a fixed orthonormal basis, say, $\{ e_n: n \in \N\}$ and
   is known as the Schur product (or the Hadamard product). Schur
   product of two operators $T$ and $S$ is denoted by $T \star S$. It
   is known that $S_p(\mathcal H)$ forms a (commutative) Banach
   algebra with respect to Schur product as well. We will discuss
   only $S_2(\mathcal H)$ in this section.

   For each $T \in S_2(\mathcal H)$, $T_{rs}:=\langle T(e_s),
   e_r\rangle$ for all $r, s \in \N$; so that $\|T\|_2^2 = \sum_{r, s
     = 1}^\infty |T_{rs}|^2$.

   \begin{lemma}\label{star-lemma}
Let $\{V^{(i)}\}$ and $\{W^{(i)}\}$ be two bounded sequences
converging weakly to $V$ and $W$, respectively, in $S_2(\mathcal
H)$. Then, $V^{(i)}\star U \longrightarrow V\star U$ in norm in
$S_2(\mathcal H)$ for every $U\in S_2(\mathcal H)$.
   \end{lemma}
   \begin{proof}
Since $V^{(i)}$ converges weakly to $V$, we have
\begin{equation}\label{limit}
V^{i}_{rs} = \langle V^{i} (e_s), e_r \rangle \to \langle V (e_s), e_r \rangle = V_{rs}
\end{equation}
for all $r, s \in \N$.

Now, let $U\in S_2(\mathcal H)$. Then, for any pair $m, n \in \N$, we have
\begin{equation*}
\begin{split}
\lim \sup_i ||V^{(i)}\star U-V\star U||_2^2 & = \lim \sup_i
\sum_{r,s}|V^{(i)}_{rs} U_{rs}-V_{rs} U_{rs}|^2 \\
& = \lim \sup_i
\sum_{r=1,s=1}^{m,n}|V^{(i)}_{rs} U_{rs}-V_{rs} U_{rs}|^2  \\ &  \qquad \qquad +
\lim \sup_i
\sum_{r>m,s>n}|V^{(i)}_{rs} U_{rs}-V_{rs} U_{rs}|^2\\
& =  \lim \sup_i \sum_{r>m,s>n}|V^{(i)}_{rs} U_{rs}-V_{rs} U_{rs}|^2 \quad \text{(by \Cref{limit})}\\
& \leq
\left(\sup_{r>m,s>n}|U_{rs}|^2\right)
\left(\sum_{r>m,s>n}|V^{(i)}_{rs}-V_{rs}|^2\right).
\end{split}
\end{equation*}
Note that $\sum_{r,s=1}^\infty |U_{rs}|^2 < \infty$. So, for every $\epsilon >
0$, there exist $m, n \in \N$ such that $\sup_{r>m,s>n}|U_{rs}^2| <
\epsilon$. And, since
\[
\left(\sum_{r>m,s>n}|V^{(i)}_{rs}-V_{rs}|^2\right) \leq
2\left( \sup_i\|V^{(i)}\|_2^2 + \|V\|_2^2\right) < \infty,
\]
we conclude that \( \lim \sup_i ||V^{(i)}\star U-V\star U||_2^2=0
\). Thus, $V^{(i)}\star U\longrightarrow V\star U$ in norm.\end{proof}

   \begin{theorem}
$S_2(\mathcal H)\otimes^\gamma S_2(\mathcal H)$ is Arens
regular if $S_2(\mathcal H)$ is equipped with Schur product.
	\end{theorem}
\begin{proof}
Let $m:\mathcal K\times \mathcal K\to \mathbb C$ be a bounded bilinear
form. Then, $m(S,T)=\left<T,\phi(S)\right>$ for all $T, S \in \mathcal
K$, for some conjugate linear continuous operator $\phi: \mathcal K
\to \mathcal K$.
		
Let $\{S^{(i)}\},\{\tilde S^{(j)}\}$ and $\{T^{(i)}\},\{\tilde
T^{(j)}\}$ be two pair of sequences in the unit ball of $\mathcal K$
such that both the iterated limits exists. 
  Through a similar reasoning as in \ref{finite-biregular}, we can assume that
these sequences converges weakly to $S,\tilde S$ and $T,\tilde T$
respectively (in weak topology of $S_2(\mathcal H)$).

By \Cref{star-lemma}, $S^{(i)}\star \tilde S^{(j)}\longrightarrow
S\star \tilde S^{(j)}$; so, $\phi(S^{(i)}\star \tilde
S^{(j)})\longrightarrow\phi(S\star \tilde S^{(j)})$ in norm.  Hence,
\[
	\lim_i m(T^{(i)}\star \tilde T^{(j)},S^{(i)}\star \tilde S^{(j)})=m(T\star \tilde T^{(j)},S\star \tilde S^{(j)}).
\]
	Similarly,
\[
	\lim_j\lim_i m(T^{(i)}\star \tilde T^{(j)},S^{(i)}\star \tilde S^{(j)})=m(T\star \tilde T,S\star \tilde S).
\]
	By a symmetric argument 
        \[
	\lim_i\lim_j m(T^{(i)}\star \tilde T^{(j)},S^{(i)}\star \tilde S^{(j)})=m(T\star \tilde T,S\star \tilde S).
	\]
	Hence, $m$ is biregular and (by \Cref{ulger} again)
          $S_2(\mathcal H)\obp S_2(\mathcal H)$ is Arens regular
          (with respect to the Schur product).
\end{proof}

\paragraph{\bf Acknowledgement}

I would like to thank Dr.~Ved Prakash Gupta for a thorough proof
reading of the manuscript and giving valuable suggestions.



\begin{thebibliography}{1}


  
  \bibitem{Arens} R.~Arens,  The adjoint of a bilinear
operation, Proc.  Amer. Math.  Soc. 26
  (1951), 839-848.

\bibitem{Arikan} N.~Arikan, A Simple Condition Ensuring the Arens
  Regularity of Bilinear Mappings, Proc. Amer.
  Math. Soc., Vol.84 (1982), 525-532.


\bibitem{Blecher} D.~P.~Blecher, Christian Le Merdy, On quotiens of
  function algebras and operator algebra structure on $\ell_p$,
  J. Operator Theory 34 (1995), no. 2  pages,

\bibitem{CB} P.~Civin and Y.~Bertram, The second conjugate space of a
  Banach algebra as an algebra, Pacific J. Math. 11 (1961), 847 - 870.

\bibitem{DL} H.~G.~Dales and A.~T.-M.~Lau, The second duals of Beurling algebras, Mem. Amer. Math. Soc. 177 (2005), no. 386. 

\bibitem{Forrest} B.~Forrest, Arens regularity and discrete groups,
  Pacific. J. Math. 151 (1991), no.2, 217 - 227.

\bibitem{GI} G.~Godefroy and B.~Iochum, Arens regularity of Banach
  algebras and the geometry of Banach spaces, J. Funct. Anal. 80
  (1988), no.1, 47 - 59.

\bibitem{KR} A.~Kumar and V.~Rajpal, Arens regularity of projective tensor products, Arch. Math. (Basel) 107 (2016), no. 5, 531 - 541.

\bibitem{LU} A.~T.-M.~Lau and A.~\"Ulger, Some geometric properties on
  the Fourier and Fourier-Stieltjes algebras of locally compact
  groups, Arens regularity and related problems,
  Trans. Amer. Math. Soc. 337 (1993), no.1, 321 - 359.
  
\bibitem{Neufang} M. Neufang, Geometry of $C^*$ -algebras, and the
  bidual of their projective tensor product, J. Funct. Anal. 278 (2019), no. 9,
  108407.

\bibitem{Palacios} A.~R.~Palacios,  A note on Arens
regularity. The Quarterly Journal of Mathematics 38 (1987), no. 1,
91-93.

\bibitem{Palmer} T.~W.~Palmer,  Banach algebras and the general
theory of $*$-algebras. Vol. 1, 2 (2001), Cambridge University press.

\bibitem{Ryan} R.~A.~Ryan,  Introduction to tensor product of
  Banach spaces. Springer monograph in mathematics.
  Trans. Amer. Math. Soc. 305 (1990), 377 - 399. 

\bibitem{Ulger} A. \"Ulger, Arens regularity of the algebra $A \obp B$,
  Trans. Amer. Math. Soc. 305 (1988),  no. 2, 623–639.


\bibitem{Ulger2} A. \"Ulger,  Weakly compact bilinear forms and Arens
regularity, Proc.  Amer.  Math. Soc. 101 (1987), no. 4,
697 -704.

\bibitem{Ulger3} A. \"Ulger, Erratum to Arens regularity of the algebra $A \obp B$,
  Trans. Amer. Math. Soc. 355 (2003), 3839.


\end{thebibliography}
\end{document}